\documentclass[12pt]{article}
\usepackage{amsmath, amssymb, amsthm, graphicx, tikz}
\usepackage[margin=1in]{geometry}
\usepackage{enumitem}
\usepackage{forest}
\usepackage{tikz}
\usetikzlibrary{trees}
\usepackage{setspace}
\usepackage{xcolor}
\usepackage{tcolorbox}
\usepackage{float}
\usetikzlibrary{backgrounds}   
\usetikzlibrary{positioning}    
\usetikzlibrary{fit}            
\usepackage[utf8]{inputenc}
\usepackage{graphicx}   
\usepackage{caption}    
\usepackage{subcaption} 
 \usepackage{geometry} % To adjust page margins for better fit if needed
\usepackage{amsthm}

\usepackage[T1]{fontenc}

\usetikzlibrary{calc}
 \definecolor{myblue}{RGB}{0,0,255}
 \usepackage{pstricks}
\usepackage{pst-node}
\usepackage{pst-plot}
 
\numberwithin{equation}{section}
\numberwithin{figure}{section}
\numberwithin{table}{section}

\newtheorem{theorem}{Theorem}[section]
\newtheorem{lemma}[theorem]{Lemma}
\newtheorem{corollary}[theorem]{Corollary}

\theoremstyle{definition}
\newtheorem{definition}[theorem]{Definition}

\theoremstyle{remark}

\theoremstyle{remark}
\newtheorem*{conjecture}{\textbf{Conjecture}}

\tikzset{
    hollow/.style={circle, draw=black, fill=white, inner sep=1pt, thick, minimum size=3pt},
    solidred/.style={circle, draw=black, fill=red, inner sep=1pt, thick, minimum size=3pt}
}

\newcommand{\pathfour}{%
\tikz[baseline=-0.5ex, scale=0.55]{%
    \node[solidred] (A) at (0,0) {};
    \node[solidred]   (B) at (0.6,0) {};
    \node[solidred]   (C) at (1.2,0) {};
    \node[solidred]   (D) at (1.8,0) {};
    \draw[black, thick] (A) -- (B) -- (C) -- (D);
}%
}

\newcommand{\pathverticalfour}{%
\tikz[baseline=-0.5ex, scale=0.55]{%
    \node[solidred] (A) at (0,-1.5) {};
    \node[hollow]   (B) at (0,-1.0) {};
    \node[hollow]   (C) at (0,-0.5) {};
    \node[hollow]   (D) at (0,0) {};
    \draw[black, thick] (A) -- (B) -- (C) -- (D);
}%
}

\newcommand{\pathLshape}{%
\tikz[baseline=-0.5ex, scale=0.55]{%
    \node[solidred] (A) at (0,-1.0) {};
    \node[hollow]   (B) at (0,0) {};
    \node[hollow]   (C) at (0.8,0) {};
    \node[hollow]   (D) at (0.8,0.8) {};
    \draw[black, thick] (A) -- (B) -- (C) -- (D);
}%
}

\newcommand{\pathHorizontalRed}{%
\tikz[baseline=-0.5ex, scale=0.55]{%
    \node[solidred] (A) at (0,-0.8) {};
    \node[hollow]   (B) at (0,0) {};
    \node[hollow]   (C) at (0.8,0) {};
    \node[hollow]   (D) at (1.6,0) {};
    \draw[black, thick] (A) -- (B) -- (C) -- (D);
}%
}

\newcommand{\pathVshape}{%
\tikz[baseline=-0.5ex, scale=0.55]{%
    \node[solidred] (A) at (0,0) {};
    \node[hollow]   (B) at (-0.8,0.8) {};
    \node[hollow]   (C) at (0.8,0.8) {};
    \node[hollow]   (D) at (1.6,1.6) {};
    \draw[black, thick] (A) -- (B);
    \draw[black, thick] (A) -- (C) -- (D);
}%
}

\newcommand{\pathVshapeHorizontal}{%
\tikz[baseline=-0.5ex, scale=0.55]{%
    \node[solidred] (A) at (0,0) {};
    \node[hollow]   (B) at (-0.8,0.8) {};
    \node[hollow]   (C) at (0.8,0.8) {};
    \node[hollow]   (D) at (1.6,0.8) {};
    \draw[black, thick] (A) -- (B);
    \draw[black, thick] (A) -- (C) -- (D);
}%
}

       \newcommand{\pathCshape}{%
       \tikz[baseline=-0.5ex, scale=0.55]{%

           \node[hollow]   (TopLeft)     at (0, 1.0) {};
           \node[hollow]   (TopRight)    at (1.0, 1.0) {};
           \node[solidred] (BottomLeft)  at (0, 0) {};
           \node[solidred] (BottomRight) at (1.0, 0) {};

           \draw[black, thick] (TopLeft) -- (TopRight);
           \draw[black, thick] (TopRight) -- (BottomRight);
           \draw[black, thick] (BottomRight) -- (BottomLeft);
       }%
       }

     \newcommand{\pathCcshape}{%
     \tikz[baseline=-0.5ex, scale=0.55]{%

         \node[hollow]   (TopLeft)     at (0, 1.0) {};
         \node[hollow]   (TopRight)    at (1.0, 1.0) {};
         \node[solidred] (BottomLeft)  at (0, 0) {};
         \node[solidred] (BottomRight) at (1.0, 0) {};

         \draw[black, thick] (TopLeft) -- (BottomLeft);
         \draw[black, thick] (BottomLeft) -- (BottomRight);
         \draw[black, thick] (BottomRight) -- (TopRight);
     }%
     }

         \newcommand{\pathCccshape}{%
         \tikz[baseline=-0.5ex, scale=0.55]{%

             \node[hollow]   (TopLeft)     at (0, 1.0) {};
             \node[hollow]   (TopRight)    at (1.0, 1.0) {};
             \node[solidred] (BottomLeft)  at (0, 0) {};
             \node[solidred] (BottomRight) at (1.0, 0) {};

             \draw[black, thick] (TopLeft) -- (TopRight);
             \draw[black, thick] (TopLeft) -- (BottomLeft);
             \draw[black, thick] (TopRight) -- (BottomRight);
         }%
         }

   \newcommand{\pathLshapeNew}{%
   \tikz[baseline=-0.5ex, scale=0.55]{% % Consistent scale for the new blue-edged series
     
       \node[hollow] (Top) at (0, 2.0) {};
       \node[hollow] (Mid) at (0, 1.0) {};
       \node[solidred] (Corner) at (0, 0) {}; 
       \node[solidred] (Right) at (1.0, 0) {};
       
       \draw[black, thick] (Top) -- (Mid);
       \draw[black, thick] (Mid) -- (Corner);
       \draw[black, thick] (Corner) -- (Right);
   }%
   }

     \newcommand{\pathZigzag}{%
     \tikz[baseline=-0.5ex, scale=0.40]{% % Consistent scale

         \node[hollow] (LeftTop) at (0, 1.0) {};
         \node[solidred] (LeftBottom) at (1.0, 0) {};
         \node[hollow] (RightTop) at (2.0, 1.0) {};
         \node[solidred] (RightBottom) at (3.0, 0) {};

         \draw[black, thick] (LeftTop) -- (LeftBottom);
         \draw[black, thick] (LeftBottom) -- (RightTop);
         \draw[black, thick] (RightTop) -- (RightBottom);
     }%
     }

        \newcommand{\pathGammaShape}{%
        \tikz[baseline=-0.5ex, scale=0.55]{% % Consistent scale

            \node[hollow] (Top) at (0, 1.0) {}; 
            \node[solidred] (Corner) at (0, 0) {}; 
            \node[solidred] (Mid) at (1.0, 0) {}; 
            \node[solidred] (Right) at (2.0, 0) {};

            \draw[black, thick] (Top) -- (Corner);
            \draw[black, thick] (Corner) -- (Mid);
            \draw[black, thick] (Mid) -- (Right);
        }%
        }

     \newcommand{\pathPeakShape}{%
     \tikz[baseline=-0.5ex, scale=0.40]{% % Consistent scale

         \node[solidred] (Left) at (0, 0) {};
         \node[hollow] (Peak) at (1.0, 1.0) {}; 
         \node[solidred] (Mid) at (2.0, 0) {}; 
         \node[solidred] (Right) at (3.0, 0) {};

         \draw[black, thick] (Left) -- (Peak);
         \draw[black, thick] (Peak) -- (Mid);
         \draw[black, thick] (Mid) -- (Right);
     }%
     }

\begin{document}
 
\title{Towards Esperet's Conjecture: Polynomial $\chi$-Bounds for Structured Graph Classes}

\author{
N. Rahimi\thanks{ORCID: 0009-0004-9571-669X}\ , D.A. Mojdeh\thanks{Corresponding author, ORCID: 0000-0001-9373-3390}\\
$^{1}$Department of Mathematics, Faculty of Mathematical Sciences,\\ University of Mazandaran, Babolsar, Iran\\
 $^{1}${\tt narjesrahimi1365@gmail.com}\\
 $^{2}${\tt damojdeh@umz.ac.ir}\\  
 $^{2}${\tt damojdeh@yahoo.com} 
}
\date{} 

\maketitle

\begin{abstract}
In this paper, we establish that the class of $\{P_6, (2,2)\text{-broom}\}$-free graphs contains a subclass $\mathcal{L}_i$, defined by certain cutset conditions, whose chromatic number admits a linear $\chi$-bound. 
Building on recent results showing that broom-free graphs excluding $K_d(t)$ as a subgraph admit a polynomial bound in~$t$ on their chromatic number (A broom is obtained from a path with one end v by adding leaves adjacent to v), we extend this result to the hereditary class $\mathcal{H}$ of $C_4$-free and \emph{$p$-flag}-free graphs (where a \emph{$p$-flag} is a triangle with an attached p-path).
We show that if $G \in \mathcal{H}$ is $B^{+}(p+2, t-1)$-free (for $p\!\ge2$ and $t\!\ge3$, that is, if it excludes a generalized broom with an additional leaf), and does not contain $K_d(t)$ as a subgraph, then $\chi(G)$ is polynomially bounded in~$t$. 
Furthermore, for the subclass of $\mathcal{H}$ excluding $K_3(t)$ as a subgraph, we prove that $\chi(G)$ is linearly $\chi$-bounded in~$\omega(G)$.
\end{abstract}

{\bf Keywords}: Graph vertex coloring, chromatic number, clique number, $\{P_6, (2,2)\text{-broom}\}$-free graphs, {$p$-flag}-free graphs.

{\bf 2020 Mathematics Subject Classification}:	05Cxx

\section{Introduction}

A \emph{$k$-coloring} of a graph is an assignment of colors from $\{1, \dots, k\}$ to its vertices such that adjacent vertices receive distinct colors. The \emph{chromatic number} $\chi(G)$ is the minimum $k$ for which such a coloring exists.
A significant challenge in graph theory lies in understanding graphs with arbitrarily large chromatic number. While a small chromatic number implies that the graph can be partitioned into a few simple, independent sets, the situation for graphs with high chromatic number is far less clear. This raises a natural question: what can be said about the local structure of such graphs? More precisely, does every graph with large chromatic number necessarily contain certain canonical substructures?

A class of graphs $\mathcal{G}$ is \emph{hereditary} if it is closed under taking induced subgraphs; that is, for every graph $G \in \mathcal{G}$, all induced subgraphs of $G$ also belong to $\mathcal{G}$. A central and extensively studied example of a hereditary graph class is the family of \emph{$H$-free} graphs, those that exclude a fixed graph $H$ as an induced subgraph. More generally, given a family of graphs $\mathcal{H}$, we say that a graph $G$ is \emph{$\mathcal{H}$-free} if $G$ is $H$-free for every $H \in \mathcal{H}$. We write $|G|$ for its number of vertices,
$\omega(G)$ for its clique number, and $\alpha(G)$ for its independence number. 

A hereditary class of graphs $\mathcal{G}$ is \emph{$\chi$-bounded} if there exists a function $f$ (called a \emph{$\chi$-binding function}) such that $\chi(G) \leq f(\omega(G))$ for every $G \in \mathcal{G}$. When $f$ can be taken to be a polynomial, $\mathcal{G}$ is said to be \emph{polynomially $\chi$-bounded}. Such graph classes are of significant interest, in part because they satisfy the Erdős--Hajnal conjecture~\cite{13}, which concerns the Ramsey number of $H$-free graphs.

\begin{conjecture}\label{con ER}
For every graph $H$, there exists a constant $c > 0$ such that
$\alpha(G) \cdot \omega(G) \geq |G|^c$
for every $H$-free graph $G$.
\end{conjecture}

Louis Esperet~\cite{Esp17} made the following conjecture:

\begin{conjecture}\label{con ESP}
Let $\mathcal{G}$ be a $\chi$-bounded class. Then there is a polynomial function $f$ such that $\chi(G) \leq f(\omega(G))$ for every $G \in \mathcal{G}$.
\end{conjecture}

Esperet's conjecture was shown false by Brianski, Davies and Walczak~\cite{BDW22}. However, this raises the further question: which $\chi$-bounded classes are polynomially $\chi$-bounded?

Gyárfás~\cite{Gyarfas1975} and Sumner~\cite{Sumner1981} independently conjectured that for every
tree~$T$, the class of $T$-free graphs is $\chi$-bounded.
This conjecture has been confirmed for some special trees
(see, for example, \cite{Chudnovsky2019,Gyarfas1975,Gyarfas1980,Kierstead1994,Kierstead2004,Scott1997,Scott2020}),
but remains open in general.

By combining the \emph{Gyárfás--Sumner Conjecture} with \emph{Esperet's Conjecture}, one obtains the following conjecture:

\begin{conjecture}\label{con MIX}
For every forest $H$, there exists $c > 0$ such that $\chi(G) \leq \omega(G)^c$ for every $H$-free graph $G$.
\end{conjecture}
  
Every forest H satisfying this conjecture also satisfies the Erdős-Hajnal conjecture.

All graphs considered in this paper are finite and simple.  
We use $P_k$ and $C_k$ to denote an induced path and an induced cycle on $k$ vertices, respectively. In this paper, a $p$-path refers to a path of length p.

Let $G$ be a graph and $X \subseteq V(G)$.  
We use $G[X]$ to denote the subgraph of $G$ induced by $X$.
For $X \subseteq V(G)$, let  
$N_G(X) = \{\, u \in V(G) \setminus X : u \text{ has a neighbor in } X \,\}$ and

$N_G[X] = N_G(X) \cup X$.

Let $\Delta(G)$ and $\delta(G)$ be the maximum and minimum degree of $G$, respectively. For $u,v \in V(G)$, we simply write $u \sim v$ if $uv \in E(G)$,  
and write $u \nsim v$ if $uv \notin E(G)$.

For a positive integer~$i$, let  
$
N^{i}(X) := \{\, u \in V(G) \setminus X : \min\{ d(u,v) : v \in X \} = i \,\},
$
where $d(u,v)$ is the distance between $u$ and $v$ in~$G$.  
Then $N^{1}(X) = N(X)$ is the neighborhood of~$X$.  
Moreover, let
$
N^{\ge i}(X) := \bigcup_{j=i}^{\infty} N^{j}(X).
$
Let $H$ be any subgraph of~$G$.  
For simplicity, we sometimes write $N^{i}(H)$ for $N^{i}(V(H))$ and $N^{\ge i}(H)$ for $N^{\ge i}(V(H))$.  
When $G$ is clear from the context, we omit the subscript.
 Given a set $\mathcal{H}$ of graphs, let $f^*_{\mathcal{H}}(\omega)$ be the optimal $\chi$-binding function for $\mathcal{H}$-free graphs, i.e., $f^*_{\mathcal{H}}(\omega) = \max\{\chi(G) : \omega(G) = \omega,\ G \text{ is } \mathcal{H}\text{-free}\}$.
 
The Ramsey number $R(p,q)$ is the least number $N$ such that every graph on $N$ vertices contains an independent set on
$p$ vertices or a clique on $q$ vertices. Erdős and Szekeres~\cite{Erdos1935} proved in 1935 that for any positive integers $s$ and $t$,
$R(s, t) \le \binom{s + t - 2}{t - 1}$. Denote the Ramsey number in the family of $F$-free graphs by $R_F(p,q)$ for convenience, we will write $R_{C_4}(p,q)$ as $\varphi(p,q)$.\\
In 2024, Zhou et al.~\cite{Zhou2024} proved the following fundamental result concerning the Ramsey number
$\varphi(n, \omega)$ in the family of $C_4$-free graphs.

\begin{theorem}\label{TH PHI}\emph{(\cite{Zhou2024}, Theorem~1.3)}
$\varphi(n, \omega)$ is linear in~$\omega$.
\end{theorem}

In the proof of Theorem~\ref{TH PHI}, the authors obtained the following explicit bounds which will be used in our work:
$\varphi(3,\omega) \le \left\lfloor \tfrac{5(\omega-1)}{2} \right\rfloor + 1$ (Claim~2.1), and 
$\varphi(n,\omega) \le \binom{n}{2}(\omega - 2) + n$ for $n>3$ and $\omega>1$ (Claim~2.3).
It is evident that for all $n \ge 3$ and $\omega \ge 1$, we have
\[
\varphi(n,\omega) \le \binom{n}{2}(\omega - 1) + n.
\]

The class of $P_4$-free graphs is particularly well-behaved: 
as shown in~\cite{24}, every $P_4$-free graph $G$ with at least two vertices is either disconnected or its complement is disconnected, 
and consequently $\chi(G) = \omega(G)$. Until now, whether there exists a polynomial bound for the chromatic number of
$P_t$-free graphs ($t \ge 5$) is still an open question. The best known result is quasi-polynomial for $P_5$-free graphs:
$f(\omega) \le \omega^{\log_2 \omega}$, proposed by Scott et~al.~\cite{p5Scott2023}. Cameron et~al.~\cite{Cameron2018} proved that for any $(P_6, \text{diamond})$-free graph $G$, 
the chromatic number satisfies $\chi(G) \le \omega(G) + 3$.

However, forbidding additional graphs may lead to a polynomial bound. The family of $C_4$-free graphs is known not to be $\chi$-bounded. Brause et~al.~\cite{Brause2022} proved that
$
f^{*}_{(P_5,C_4)}(\omega) \le \left\lceil \frac{5\omega - 1}{4} \right\rceil.
$

\begin{definition}\label{DEF BTK}
For positive integers $t,k$, a $(t,k)$-broom is the graph obtained from $K_{1,t+1}$ by subdividing an edge $k$ times. In some papers, a $(t,1)$-broom is called a \emph{$t$-broom}, and a $(t,2)$-broom is called a \emph{$t$-broom$^+$} (Figure 2.1(d)).  
\end{definition}
Chudnovsky et~al.~\cite{Chudnovsky2021} proved that each $(C_4, 2\text{-broom})$-free graph $G$ satisfies
$
\chi(G) \le \frac{3}{2} \,\omega(G)
$
and Several results in this direction are presented below:
\begin{theorem}\label{TH OM T+1}\emph{(\cite{Liu2023} 1.1)}
Let $t$ be a positive integer. For $t$-broom-free graphs $G$, we have $\chi(G) = o\big(\omega(G)^{t+1}\big).$
\end{theorem}

\begin{theorem}\label{TH T,T}\emph{(\cite{Liu2023} 1.3)}
Let $t \ge 3$ be an integer.  
For all $\{t\text{-broom},\, K_{t,t}\}$-free graphs $G$, we have
$
\chi(G) = o\big(\omega(G)^{t}\big).
$
\end{theorem}

\begin{theorem}\emph{(~\cite{Zhou2024} 1.4)}\label{TH CT,T}
Let  $t \ge 2$ be an integer and let $\mathcal{G}$ be the family of $(C_4,\, t\text{-broom}^+)$-free graphs.  
Then $\mathcal{G}$ is quadratically $\chi$-bounded.
\end{theorem}

While Esperet’s Conjecture remains open, for a graph $G$ and an integer $d \ge 1$, it may be preferable to replace a clique with a complete $d$-partite graph, each part of cardinality $t$, denoted by $K_d(t)$. This question has been posed in \cite{Doe2024}.

% ==========================
An analogue of Esperet’s (false) conjecture for $K_d(t)$:\\
Let $\mathcal{C}$ be a hereditary class of graphs, and let $d \ge 1$. Suppose that $\mathcal{C}$ does not contain $K_d(t)$ as a subgraph (not necessarily induced). Assume there exists a function $f$ such that
\[
\chi(G) \le f(t)
\]
for each $G \in \mathcal{C}$.  
Can we always choose $f$ to be a polynomial?

  If \( X \subseteq V(G) \), we define \( \chi(X) = \chi(G[X]) \). Let \( X, Y \) be disjoint subsets of \( V(G) \), where \( |X| = t \) and every vertex in \( Y \) is adjacent to every vertex in \( X \). We call \( (X, Y) \) a \textbf{\( t \)-biclique}, and its value is \( \chi(Y) \).

For integers \( p, t \geq 1 \), a \textbf{\( (p, t) \)-balloon} \( (P, Y) \)(or simply a $(p,t)$-balloon $Y$) in \( G \) consists of an induced path \( P \) in \( G \) with vertices \( v_1\text{-}v_2\text{-}\cdots\text{-}v_p \) (in order), and a subset \( Y \subseteq V(G) \), satisfying the following properties:
\begin{itemize}[leftmargin=*, label=$\bullet$]
    \item \( v_1, \ldots, v_{p-1} \notin Y \), \( v_p \in Y \). None of \( v_1, \ldots, v_{p-2} \) have a neighbor in \( Y \), and if \( p \geq 2 \), then \( v_p \) is the unique neighbor of \( v_{p-1} \) in \( Y \).
    \item \( G[Y] \) is \( t \)-connected.
\end{itemize}
The value of the \( (p, t) \)-balloon \( (P, Y) \) is \( \chi(Z) \), where \( Z \) is the set of vertices in \( Y \) nonadjacent to \( v_p \) (hence \( v_p \in Z \)).
\vspace{2mm}
  
\begin{lemma} \emph{(\cite{Tu2023} 2.3)}\label{lem-mgun}
For every graph $G$, and all integers $p, q, s, t \geq 1$, if $G$ contains no $(p,t)$-balloon of value at least $q$, and $G$ contains no $t$-biclique of value $s$, then:
\[
\chi(G) \leq (\sum_{i=0}^{p-1} t^i)(s + t(2t + 9)) + t^p q
\]
\end{lemma}
  
\begin{theorem}\emph{(\cite{Tu2023} 2.4)}\label{TH D PATH}
For every path $H$ and all $d \ge 1$, there is a polynomial $f$ such that, for all $t \ge 1$,  
if a graph $G$ is $H$-free and does not contain $K_d(t)$ as a subgraph, then $\chi(G) \le f(t)$.
\end{theorem}

\begin{theorem}\emph{(\cite{Tu2023} 3.2)}\label{theotu23-2}
For all $d \ge 1$ and every broom $H$ 
(A broom is obtained from a path with one end v by adding leaves adjacent to v), there is a polynomial $f$ such that, for all $t \ge 1$,  
if a graph $G$ is $H$-free and does not contain $K_d(t)$ as a subgraph, then $\chi(G) \le f(t)$.
\end{theorem}

A graph $G$ is called \emph{$d$-degenerate} if every nonnull subgraph has a vertex of degree at most $d$. The degeneracy $\partial(G)$ is the smallest such $d$, Then $\chi(G) \leq \partial(G) + 1$.\\
A \emph{rooted tree} $(H, r)$ consists of a tree $H$ and a vertex $r$ of $H$ called the root. The \emph{height} of $(H, r)$ is the length (number of edges) of the longest path in $H$ with one endpoint $r$. The \emph{spread} of $H$ is the maximum over all vertices $u \in V(H)$ of the number of children of $u$.
Let $\zeta, \eta \ge 1$. The rooted tree $(H, r)$ is $(\zeta, \eta)$-uniform if every vertex with a child has exactly $\zeta$ children;
every vertex with no child is joined to $r$ by a path of $H$ of length exactly $\eta$.

\begin{theorem}\label{theo-scsys} \emph{(\cite{Scott2023} 3.2)}
Let $\eta, t \geq 1$ and $\zeta \geq 2$. For every rooted tree $(H,r)$ with height at most $\eta$ and spread at most $\zeta$, let $c = (\eta + 3)! |H|$. Then $\partial(G) \leq (|H| \zeta t)^c$ for every $H$-free graph $G$ that does not contain $K_{t,t}$ as a subgraph. 
\end{theorem}

\section{Main results}

We define the following trees and graph classes, which are fundamental to our main results:

\begin{definition}\label{DEF S+}
A \textit{2-arm star} is an induced tree that can be regarded as a subdivision of a star graph with two arms. 
It has a central vertex ($u_0$) which, for $t \ge 5$, is adjacent to $t - 4$ pendant vertices, 
and two additional induced paths incident with $u_0$: one of length $4$ and the other of length $p$ ($p \ge 1$). 
This tree is denoted by $S(1^{(t-4)}, 4, p)$.
A \emph{2-arm star tree}, denoted by $ S(1^{(t-4)}, 4, p)$ (Figure 2.1(b)).
\end{definition}

\begin{definition}\label{DEF B+}

$B^{+}(p+2, t-1)$ is an induced tree resembling a broom, 
where the central vertex $u_0$ has $t-1$ pendant leaves for $t \ge 3$, 
and there exists a path of length $p+2$ for $p \ge 2$. 
In addition, at the vertex ($u_2$), which is at distance $2$ from the central vertex, 
there is one additional pendant vertex ($v$) attached (Figure 2.1(c)).

\end{definition}

\begin{definition}\label{DEF H}
We define the class $\mathcal{H}$ to be the family of graphs that are both $C_4$-free and \emph{$p$-flag}-free, 
where a \emph{$p$-flag} is an induced graph obtained from a cycle $C_3$ 
by attaching a path of length $p$ ($p \ge 1$) to one vertex of the $C_3$ (Figure 2.1(a)).
\end{definition}

\begin{figure}[H]
    \centering
    % ردیف اول
    \begin{minipage}[b]{0.45\textwidth}
        \centering
        \begin{tikzpicture}[scale=1,
            every node/.style={circle, draw, fill=white, inner sep=1.5pt},
            vertex/.style={circle, draw, fill=black!10, inner sep=1.5pt},
            label style/.style={font=\small}]
            
            \begin{scope}
                \node (v1) at (-1, 0) {};
                \node (v2) at (0, 0.5) {};
                \node (v3) at (0, -0.5) {}; 
                \draw (v1) -- (v2) -- (v3) -- (v1);
                \draw[line width=1pt, red, smooth, tension=1] 
                    plot coordinates {(-3.5, 0) (-3, 0.2) (-2.5, -0.2) (-2, 0.1) (-1.5, 0) (v1)};
                \node[draw=none, fill=none, red] at (-2, 0.7) {p-path};
               
            \end{scope}
        \end{tikzpicture}
        \subcaption{\emph{$p$-flag}}
        \label{fig:sub1}
    \end{minipage}
    \hfill
    \begin{minipage}[b]{0.45\textwidth}
        \centering
        \begin{tikzpicture}[scale=1.1,
            vertex/.style={circle, draw, fill=black!10, inner sep=1.5pt}]
            
            \begin{scope}
                \node (u0) at (0,0) [vertex, label=above:$u_0$] {};
                \node (v1) at (-0.5,-0.8) [vertex, label=below:$v_1$] {};
                \node (v2) at (0.0,-0.8) [vertex, label=below:$v_2$] {};
                \node (vtm4) at (0.6,-0.8) [vertex, label={[yshift=-20pt, xshift=5pt]$v_{t-4}$}] {};
                \node[draw=none, fill=none] at (0.3,-0.8) {$\cdots$};
                \draw (u0)--(v1); \draw (u0)--(v2); \draw (u0)--(vtm4);

                \node (a1) at (1.1,0) [vertex, label=below:$u_1$] {};
                \node (a2) at (2.0,0) [vertex, label=below:$u_2$] {};
                \node (a3) at (2.9,0) [vertex, label=below:$u_3$] {};
                \node (a4) at (3.8,0) [vertex, label=below:$u_4$] {};
                \draw (u0)--(a1)--(a2)--(a3)--(a4);

                \draw[line width=1pt, red, smooth, tension=1] 
                    plot coordinates {(-1.8,0) (-1.3,0.2) (-0.8,-0.2) (-0.3,0.1) (u0)};
                \node[draw=none, fill=none, red] at (-1,0.7) {p-path};
              
            \end{scope}
        \end{tikzpicture}
        \subcaption{$S(1^{(t-4)}, 4, p)$}
        \label{fig:sub2}
    \end{minipage}
    
    \vspace{1cm}
    
    % ردیف دوم
    \begin{minipage}[b]{0.45\textwidth}
        \centering
        \begin{tikzpicture}[scale=1.2,
            vertex/.style={circle, draw, fill=black!10, inner sep=1.5pt}]
            
            \begin{scope}
                \node (u0) at (3.0, 0) [vertex, label=below:$u_0$] {};
                \node (u1) at (2.0, 0) [vertex, label=below:$u_1$] {};
                \node (u2) at (1.0, 0) [vertex, label=below:$u_2$] {}; 
                \draw (u0) -- (u1) -- (u2);

                \node (v1) at (3.6, 0.6) [vertex, label={[xshift=6pt, yshift=-8pt]$v_1$}] {};
                \node (v2) at (3.6, 0.0) [vertex, label={[xshift=6pt, yshift=-8pt]$v_2$}] {}; 
                \node[draw=none, fill=none] (dots_v) at (3.6, -0.4) {$\vdots$}; 
                \node (vt_1) at (3.6, -0.8) [vertex, label={[xshift=12pt, yshift=-10pt]$v_{t-1}$}] {}; 

                \draw (u0) -- (v1); \draw (u0) -- (v2); \draw (u0) -- (vt_1);

                \node (v_extra) at (1.0, 0.7) [vertex, label=above:$v$] {};
                \draw (u2) -- (v_extra);

                \draw[line width=1pt, red, smooth, tension=1] 
                    plot coordinates {(-1.2, 0) (-0.7, 0.2) (-0.3, -0.2) (0.2, 0.1) (0.6, 0) (u2)};
                \node[draw=none, fill=none, red] at (0, 0.7) {p-path};
              
            \end{scope}
        \end{tikzpicture}
        \subcaption{$B^{+}(p+2, t-1)$}
        \label{fig:sub3}
    \end{minipage}
    \hfill
    \begin{minipage}[b]{0.45\textwidth}
        \centering
        \begin{tikzpicture}[scale=0.8,
            vertex/.style={circle, draw, fill=black!10, inner sep=1.5pt}]
            
            \begin{scope}
                \node (u0) at (0,0) [vertex, label=above:$u_0$] {};
                \node (u1) at (-1.5,0) [vertex, label=above:$u_1$] {};
                \node (u2) at (-3,0) [vertex, label=above:$u_2$] {};
                \node (u3) at (-4.5,0) [vertex, label=above:$u_3$] {};

                \node (v1) at (0.8,1.5) [vertex, label=above:$v_1$] {};
                \node (v2) at (0.8,0.5) [vertex, label=above:$v_2$] {};
                \node (vt) at (0.8,-1.5) [vertex, label=above:$v_t$] {};
                
                \coordinate (vdots_center) at (0.8,0); 
                \draw[dotted, thick] ([yshift=-0.2cm] vdots_center) -- ([yshift=0.2cm] vdots_center);

                \draw (u0) -- (u1) -- (u2) -- (u3);
                \draw (u0) -- (v1); \draw (u0) -- (v2); \draw (u0) -- (vt);
            \end{scope}
        \end{tikzpicture}
        \subcaption{(t,2)-broom}
        \label{fig:sub4}
    \end{minipage}
    
    \caption{}
    \label{fig:all_figures}
\end{figure}

\subsection{\(\{P_6, (2,2)\text{-broom}\}\)-free}

We define the following classes.\\

Let \( \mathcal{M} \) be a hereditary class of \(\{P_6, (2,2)\text{-broom}\}\)-free graphs with binding function \( f \). For every \( G \in \mathcal{M} \), let \( X \) be a minimal cutset of \( G \) containing some vertex \( v \) (possibly after replacing \( v \) with another vertex).

The class \( \mathcal{L}_i \subseteq \mathcal{M} \), where $i \in \mathbb{N}_{\geq 2}$, is defined as follows.

\medskip
\noindent
\textbf{Case 1.} If $X = K_m$ with $m \geq 1$, then there exists a component $B_1$ of $G \setminus X$ 
and a vertex $y \in N_{B_1}(v)$ such that $B_1$ contains a subgraph $F$ satisfying:
\begin{enumerate}
    \item $V(F) \nsubseteq N_{B_1}(v)$
    \item $N_{B_1}(y) \cap V(F) \neq \emptyset$
    \item $\chi(F) > f\!\left(\left\lfloor \tfrac{\omega(G)}{i} \right\rfloor\right)$ (Although $\chi(F) \le f(\omega(F)) \le f(\omega(G))$ holds, for every subgraph $F$ there exists a sufficiently large $j \in \mathbb{N}_{\geq 2}$ such that $\chi(F) > f\!\left(\left\lfloor \tfrac{\omega(G)}{j} \right\rfloor\right)$)
\end{enumerate}
Moreover, there exists another component $B_2$ such that for some vertex $u \in B_2$ we have 
$v \notin N_X(u)$.

\medskip
\noindent
\textbf{Case 2.} If $X \neq K_m$ with $m \geq 1$, then there exists a vertex $w \in X$ nonadjacent to $v$, 
and a component $B_1$ of $G \setminus X$ with a vertex $y$ such that 
$y \in N_{B_1}(v)$ and $y \notin N_{B_1}(w)$.  
Furthermore, $B_1$ contains a subgraph $F$ satisfying the three conditions listed above.\\

The two cases above stem from the fact that either there exists a vertex in $B_2$ that is not adjacent to any vertex of $X$, or no such vertex exists.
It should be noted that it is clear that the simultaneous existence of vertices $u$ and $w$ does not occur. 
Moreover, if the vertex $f$ in $B_2$ does not exist to form the induced path $v\!-\!f\!-\!w$, 
there always exists a minimal induced path between $v$ and $w$ through $B_2$ of length greater than 2 (Figure 2.2).

\begin{theorem}\label{TH 00}
For the class $\mathcal{L}_i$, where $i \in \mathbb{N}_{\geq 2}$, there exists a binding function whose growth is at most linear.
\end{theorem}

\begin{proof}
Assume, for the sake of contradiction, that every binding function for the class $\mathcal{L}_i$ has superlinear growth. Let $f$ be such a binding function for $\mathcal{L}_i$. let \( X \) is a minimal cutset of
\( G \); so \( G \setminus X \) has at least two components, and every vertex in \( X \) has a neighbour in \( V(B_t) \),
for every component \( B_t \) of \( G \setminus X \).
Choose \( v \in X \), and let \( N \) be the set of vertices in \( B_1 \) adjacent to \( v \).  
Let the components of \(  B_1  \setminus N \) be \( F_1, \ldots, F_k, Q_1, \ldots, Q_\ell \), where \( F_1, \ldots, F_k \) each have chromatic number more than \( f(\lfloor \omega(G)/i \rfloor) \), and \( Q_1, \ldots, Q_\ell \) each have chromatic number at most \( f(\lfloor \omega(G)/i \rfloor) \).
Let \( Y_j \) be the set of vertices in \( N \) with a neighbour in \( V(F_j) \), and according to the definition of the class $\mathcal{L}_i$, without loss of generality, assume that $F = F_1$.
we will prove 1 and 2.

\begin{enumerate}[label={(\arabic*)}]
    \item All vertices in \( Y_1 \) are adjacent to \( F_1 \).\\
    Let \( y_1 \in Y_1 \). Thus, \( y_1 \) has a neighbour in \( V(F_1) \); suppose that \( y_1 \) is not adjacent to some vertex in \( F_1 \). Since \( F_1 \) is connected, there is an edge \( ab \) of \( F_1 \) such that \( y_1 \) is adjacent to \( a \) and not to \( b \). Since \( G \in \mathcal{L}_i \), the vertices \( \{a, b, y_1, v, f, u\} \)(or \( \{a, b, y_1, v, f, w\} \)) (as illustrated in Figure~\ref{fig:classL}) would induce a \( P_6 \), a contradiction.

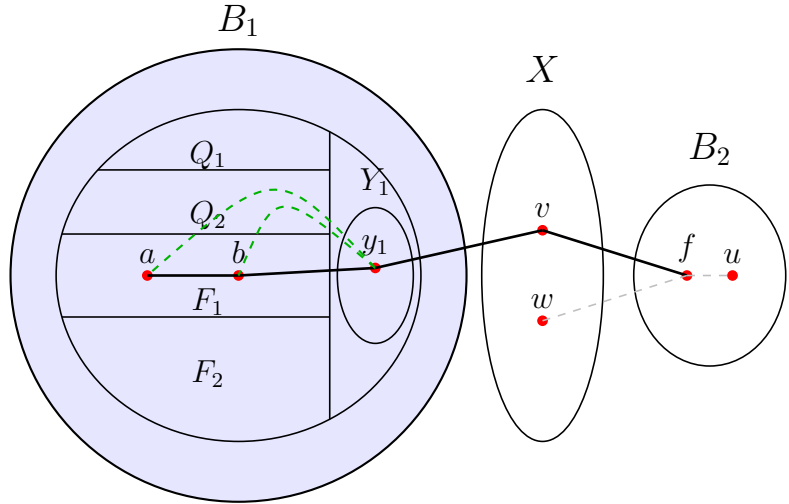
\begin{figure}[h!]
    \centering
    \begin{tikzpicture}[
        scale=1.0,
        every node/.style={font=\normalsize, inner sep=1pt, outer sep=1pt},
        line width=0.6pt
    ]

        % Outer circle for B1
        \begin{scope}[shift={(-4,0)}]
            \fill[blue!10] (0,0) circle (3cm); % Outer B1 circle
            \draw[thick] (0,0) circle (3cm);
            \node[above=3.1cm, font=\large\bfseries] at (0,0) {$B_1$};

            % Inner B1 component
            \draw (0,0) ellipse (2.4cm and 2.2cm);

            % Internal segments
            \pgfmathsetmacro{\yvertmax}{sqrt(2.2^2*(1-(1.2^2/2.4^2)))}
            \draw (1.2cm, \yvertmax) -- (1.2cm, -\yvertmax);

            \pgfmathsetmacro{\xleftsone}{sqrt(2.4^2*(1-(1.4^2/2.2^2)))}
            \draw (-\xleftsone, 1.4cm) -- (1.2cm, 1.4cm);
            \node at (-0.4cm, 1.6cm) {$Q_1$};

            \pgfmathsetmacro{\xleftsj}{sqrt(2.4^2*(1-(0.55^2/2.2^2)))}
            \draw (-\xleftsj, 0.55cm) -- (1.2cm, 0.55cm);
            \node at (-0.4cm, 0.8cm) {$Q_2$};

            \pgfmathsetmacro{\xleftri}{sqrt(2.4^2*(1-(-0.55)^2/2.2^2))}
            \draw (-\xleftri, -0.55cm) -- (1.2cm, -0.55cm);
            \node at (-0.4cm, -0.35cm) {$F_1$};

            % Define coordinates for points
            \coordinate (a) at (-1.2cm, 0.0cm);
            \coordinate (b) at (0.0cm, 0.0cm);
            \coordinate (y1) at (1.8, 0.1);

            % Colored points
            \fill[red] (a) circle (2pt);
            \node[above=0.1cm] at (a) {$a$};

            \fill[red] (b) circle (2pt);
            \node[above=0.1cm] at (b) {$b$};

            \fill[red] (y1) circle (2pt);
            \node[above=0.1cm] at (y1) {$y_1$};

            % Lines inside B1
            \draw[line width=1pt]  (b) -- (y1);
            \draw[line width=1pt] (a) -- (b);

            % R2 label
            \node at (-0.4cm, -1.3cm) {$F_2$};

            % Small ellipse Y
            \draw (1.8cm, 0) ellipse (0.5cm and 0.9cm);
            \node[above=1.0cm] at (1.8cm, 0.0) {$Y_1$};

            % Dashed smooth connections from y1 to a and b (green, from above)
            \draw[dashed, green!70!black, line width=0.8pt] 
                (y1) .. controls (0.5,1.5) .. (a);
            \draw[dashed, green!70!black, line width=0.8pt] 
                (y1) .. controls (0.5,1.2) .. (b);
        \end{scope}

        % X component
        \begin{scope}[shift={(0,0)}]
            \draw (0,0) ellipse (0.8cm and 2.2cm);
            \node[above=2.5cm, font=\large\bfseries] at (0,0) {$X$};

            \fill[red] (0, 0.6cm) circle (2pt);
            \node[above=0.1cm] at (0, 0.6cm) {$v$};
            \coordinate (v) at (0, 0.6);

            \fill[red] (0, -0.6cm) circle (2pt);
            \node[above=0.1cm] at (0, -0.6cm) {$w$};
            \coordinate (w) at (0, -0.6);
        \end{scope}

        % B2 component
        \begin{scope}[shift={(2.2,0)}]
            \draw (0,0) ellipse (1.0cm and 1.2cm);
            \node[above=1.4cm, font=\large\bfseries] at (0,0) {$B_2$};

            \fill[red] (0.3cm, 0) circle (2pt);
            \node[above=0.1cm] at (0.3cm, 0) {$u$};

            \fill[red] (-0.3cm, 0) circle (2pt);
            \node[above=0.1cm] at (-0.3cm, 0) {$f$};
            \coordinate (f) at (-0.3, 0);

            % Dashed line
            \draw[dashed, gray!50] (-0.3cm,0) -- (0.3cm,0);
        \end{scope}

        % Connecting edges
        \draw[line width=1pt] (y1) -- (v);
        \draw[line width=1pt] (v) -- (f);
        \draw[dashed, gray!50] (w) -- (f);

    \end{tikzpicture}
    \caption{A schematic representation of a graph in the class $\mathcal{L}_i$.}
    \label{fig:classL}
\end{figure}
    \item The subgraph \( F_1 \) does not contain \( S_2 \) as an induced subgraph.\\
Suppose that \( S_2 \) occurs in \( F_1 \)(\(S_k\) is a star with \(k+1\) vertices). Then, by (1), a \( (2,2) \)-broom would occur, which is a contradiction.
We observe that $\chi(F_1)=\omega(F_1)$, because the \( S_2 \)-free property forces the induced subgraph $F_1$ to be complete.
\end{enumerate}
 we have \( f\!\left(\left\lfloor \tfrac{\omega(G)}{i} \right\rfloor\right) < \chi(F_1)=\omega(F_1)\le \omega(G) \), and therefore \( f(\omega(G)) < i \omega(G) \).
This contradicts the superlinear growth assumption, completing the proof.
\end{proof}

\subsection{The Class that does not contain $K_d(t)$ as a subgraph}

\subsubsection{$B^{+}(p+2, t-1)$-free }

\begin{lemma}\label{lem ba B+}
Let $t \ge 3$ and $p\ge 2$. If \(G \in \mathcal{H}\) is \(B^{+}(p+2, t-1)\)-free, then for every $(p, t)$-balloon, \( (P, Y) \) of $G$ with the common vertex $v_p$ between the path $P$ and $Y$, we have $\Delta(Y[N^{\geq 2}(v_p)]) < R(t, \omega)$, where $\omega = \omega(G)$.
\end{lemma}

\begin{proof}
If \(G\) is \(B^{+}(p+2, t-1)\)-free. Consider a $(p,t)$-balloon $Y$ with $(a_1, \dots, a_{t'})\in N(v_p)$, where $t' \ge t$(Because the balloon $Y$ is $t$-connected).
Suppose to the contrary that there exists $z_k \in N^k(v_p)$ for some $k \geq 2$ such
that $|N(z_k) \cap N^{\geq 2}(v_p)| \geq R(t, \omega)$. Then it is clear that the induced path $z_k z_{k-1} \dots a_i v_p$ exists.

By the definition of a balloon and (1) since \(G\) is  
\emph{$p$-flag}-free, we have \(a_i \nsim a_j\) for all \(i , j\).
(2) Because \(G\) is \(C_4\)-free, every vertex in \(N^{2}(v_p)\) is adjacent to at most one vertex of \(N^1(v_p)\); and no vertex \(a_i\) can be adjacent to two vertices \(u,v\) such that the path \(u w v\) lies entirely within the second layer.  
(3) It is also clear that within \(N^{\ge 2}(v_p)\) there are no adjacencies between vertices whose layers differ by two or more.\\
Let $B = (b_1, \dots, b_t)$ be an independent set in $Y\left[(N(z_k) \cap N^{\geq 2}(v_p)\right]$.
\[
|N(z_k) \cap N^{\geq 2}(v_p)| = |(N(z_k) \cap N^{\ge 2}(v_p)) \cap N(z_{k-2})| + |(N(z_k) \cap N^{\geq 2}(v_p)) \setminus N(z_{k-2})|
\]

Due to the $C_4$-free property, $|(N(z_k) \cap N^{\geq 2}(v_p)) \cap N(z_{k-2})| = 0$. Therefore, $|(N(z_k) \cap N^{\ge 2}(v_p)) \setminus N(z_{k-2})| \ge R(t,\omega)$.

$|(N(z_k) \cap N^{\ge 2}(v_p)) \setminus N(z_{k-2})| = |(N(z_k) \cap N^{\geq 2}(v_p) \setminus N(z_{k-2})) \cap N(z_{k-1})| + |(N(z_k) \cap N^{\ge 2}(v_p) \setminus N(z_{k-2})) \setminus N(z_{k-1})|$

Due to the \emph{$p$-flag}-free property, $|(N(z_k) \cap N^{\geq 2}(v_p) \setminus N(z_{k-2})) \cap N(z_{k-1})| = 0$. Therefore, $|(N(z_k) \cap N^{\geq 2}(v_p) \setminus N(z_{k-2})) \setminus N(z_{k-1})| \geq R(t, \omega)$

So $B = (b_1, \dots, b_t)$ is an independent set in 
$Y\left[(N(z_k) \cap N^{\geq 2}(v_p) \backslash N(z_{k-2})) \backslash N(z_{k-1})\right]$.

We consider adjacencies to specific vertices without loss of generality (by shift-invariance of labeling).

\noindent \textbf{Case 1: \( z_k \in N^2(v_p) \).} \\
Suppose \( z_2 \sim a_1 \) (so \( z_2 \nsim a_i \) for \( i \ne 1 \) by Property 2). Then \( (b_i)z_2a_1v_p \) is induced. 
If some \( b_i \in N^3(v_p) \), then \( b_i \nsim a_j \). 
If some \( b_i \in N^2(v_p) \), then since no \( b_i \) is adjacent to two \( a_j \) and no \( a_j \) is adjacent to two \( b_i \) (by \( C_4 \)-free), there is a bijection between the \( b_i \) in \( N^2(v_p) \) and the \( a_j \). 
Thus, even if all \( a_j \) are adjacent to some \( b_i \), by deleting one \( b_i \in N^2(v_p) \) adjacent to \( a_2 \), we obtain the induced \( (b_1,\dots,b_{t-1})z_2a_1v_p a_2v_{p-1}v_{p-2}...v_{1}\).

\noindent \textbf{Case 2: \( z_k \in N^3(v_p) \).} \\
Suppose \( z_2 \sim a_1 \). Then \( (b_i)z_3z_2a_1 \) is induced. 
Since \( t \ge 3 \), \( a_1 \) has another neighbor \( f_2 \in N^2(v_p) \). 
Clearly \( f_2 \nsim z_2, z_3 \) (by \( C_4 \)-free,\emph{$p$-flag}-free). 
If \( f_2 \sim b_3 \in N^3(v_p) \), then \( f_2 \nsim \) \( b_{i\neq 3} \in N^2(v_p) \cup N^3(v_p) \) (by \( C_4 \)-free). 
If \( f_2 \sim b_2 \in N^2(v_p) \), then for \( b_{i\neq2} \in N^3(v_p) \), \( f_2 \nsim b_i \) (else \( f_2b_2z_3b_i \) is a \( C_4 \)). 
Similarly, for \( b_{j\neq 2} \in N^2(v_p) \), \( f_2 \nsim b_j \) (else \( f_2b_2z_3b_j \) is a \( C_4 \)).
Hence we find the induced \( (b_1,\dots,b_{t-1})z_3z_2 f_2 a_1v_p...v_1 \).

\noindent \textbf{Case 3: \( z_k \in N^4(v_p) \).} \\
The induced \( (b_i)z_4z_3z_2 \) exists. Suppose \( z_2 \sim a_1 \). 
Since \( t \ge 3 \), \( z_2 \) has another neighbor \( f_3 \in N^3(v_p) \) or \( f_2 \in N^2(v_p) \). 
If \( z_2 \sim f_3 \in N^3(v_p) \) and \( f_3 \sim b_3 \in N^4(v_p) \cup N^3(v_p) \), then \( z_2 \nsim \) \( b_{i\neq3} \in N^4(v_p) \cup N^3(v_p) \) (by \( C_4 \)-free). 
If \( z_2 \sim f_2 \in N^2(v_p) \) and \( f_2 \sim b_1 \in N^3(v_p) \), then \( f_2 \nsim \) \( b_{i\neq1} \in N^3(v_p) \) (by \( C_4 \)-free). 
Hence we find the induced \( (b_1,\dots,b_{t-1})z_4z_3z_2 f_i a_1v_p...v_1 \).

\noindent \textbf{Case 4: \( z_k \in N^{\ge 5}(v_p) \).} \\
Suppose \( z_2 \sim a_1 \).The induced \( (b_i)z_5z_4z_3z_2a_1 \) exists.  
If the third neighbor of \( z_3 \) belongs to \( N^2(v_p) \), then it is non-adjacent to \( z_{i\neq3}\) (by \( C_4 \)-free and \emph{$p$-flag}-free) and to the other mentioned vertices. 
If \( z_3 \sim f_3 \in N^3(v_p) \), then \( f_3 \) can only be adjacent to one \( b_i \in N^4(v_p) \). 
If \( z_3 \sim f_4 \in N^4(v_p) \) and \( f_4 \sim  b_1 \in N^4(v_p) \), then \( f_4 \nsim  b_{i\neq1} \in N^4(v_p) \cup N^5(v_p) \) (by \( C_4 \)-free). 
Thus we obtain the induced \( (b_1,\dots,b_{t-1})z_5z_4z_3f_i z_2a_1...v_1 \). 
For higher layers the pattern continues, yielding a \(B^{+}(p+2, t-1)\) tree.
\end{proof}
  
From the fact that \(\chi(N^{\ge 2}(v_p)) \le \Delta(N^{\ge 2}(v_p)) + 1\) and by the definition of the balloon value, $\chi(N^{\ge 2}(v_p)) + \chi(v_p)$, then every $(p,t)$-balloon $Y$ in G has a maximum value of $R(t,\omega) + 1$. we obtain the following result.

\begin{corollary}\label{COR ba B+}
Let $t \ge 3$ and $p\ge 2$. 
If $G \in \mathcal{H}$ is $B^{+}(p+2, t-1)$-free then there is no $(p,t)$-balloon in $G$ with value $R(t,\omega)+2$ ($\varphi(t, \omega)+2$).
\end{corollary}

\vspace{0.5em}

\begin{theorem}\label{TH 1 B+}
For all $d \geq 1$, $p\ge 2$, and every $B^{+}(p+2, t-1)$, there exists a polynomial $f$ such that for all $t \ge 3$, if a graph \(G \in \mathcal{H}\) does not contain $K_d(t)$ as a subgraph and \(B^{+}(p+2, t-1)\)-free, then we have $\chi(G) \le f(t)$.
\end{theorem}

\begin{proof}
We prove the theorem by induction on $d$. If $d = 1$, then the result holds trivially since graphs without $K_1(t)$ have fewer than $t$ vertices. Assume $d > 1$ and the result holds for $d-1$ and let $g(t)$ be the corresponding polynomial.
Define
\[
f(t) = (\sum_{i=0}^{p-1} t^i)\left(g(t) + 1 + t(2t + 9)\right) + t^{p} (\binom{t}{2}(dt - 1) + t +2).
\]

Let \(G \in \mathcal{H}\) and \(B^{+}(p+2, t-1)\)-free and does not contain $K_{d}(t)$. Then we  show that $\chi(G) \leq f(t)$.
The graph $G$ cannot have a balloon of size 
$\binom{t}{2}(dt - 1) + t + 2$, since $\omega < dt$ and by Theorem~\ref{TH PHI}, 
it follows that $G$ has a balloon of size $\varphi(t, \omega) + 2$, 
which contradicts the result of~\ref{COR ba B+}.

 $G$ contains no $t$-biclique of value $g(t) + 1$. since If $(X, Y)$ is such a $t$-biclique, from the inductive hypothesis $G[Y]$ contains $K_{d-1}(t)$, hence $G$ contains $K_{d}(t)$, a contradiction.

So by lemma~\ref{lem-mgun}, we obtain
\[
f(t) = (\sum_{i=0}^{p-1} t^i)\left(g(t) + 1 + t(2t + 9)\right) + t^{p} (\binom{t}{2}(dt - 1) + t +2).
\]

Thus the proof is completed.
\end{proof}
  
\subsubsection{$S(1^{(t-4)},4 ,p)$-free }

\begin{lemma}\label{lem ba S+}
Let $d \geq 1$,\(t \ge 5,p \ge 1\) and Assume that \(G \in \mathcal{H}\) is \(S(1^{(t-4)}, 4, p)\)-free.  
If \(G\) does not contain \(K_d(t)\) as a subgraph, then every \((p,t)\)-balloon \(Y\) has a maximum value of \(dt + 2\).
\end{lemma}

\begin{proof}
Consider a \((p,t)\)-balloon \(Y\) with the common vertex \(v_p\) between the path \(P\) and \(Y\) and  $(a_1, \dots, a_{t'})\in N(v_p)$, where $t' \ge t$.  
We show that \(\chi(Y[N^{\ge 2}(v_p)]) \le dt\).
Suppose, for contradiction, that \(\chi(Y[N^{\ge 2}(v_p)]) > dt\). 
 
Since $\omega < dt$, we have $\chi(Y[N^{\ge 2}(v_p)]) > \omega$.  
Given that $P_4$-free graphs are well-known to satisfy $\chi(G) = \omega(G)$, it follows that an induced path $P_4$, $u_1$-$u_2$-$u_3$-$u_4$, exists in $Y[N^{\ge 2}(v_p)]$.

Let \(j \ge 1\) be the first layer below \(P_4\) containing a vertex \(s_j\) adjacent to a vertex of \(P_4\) in layer \(j+1\) (If \(s_j \in N(v_p)\) then there exists at least one vertex of $P_4$ that belongs to layer~2 ).

\begin{enumerate}
    \item If there is only one vertex of \(P_4\) in layer \(j+1\), say \(u_l\), then \(s_j \sim u_l\).  
    In any case, there exist two non-adjacent vertices in \(P_4\) in layer \(j+2\) or above, so an induced path \(s_j-u_l-u_{l+1}-u_{l+2}\) exists and for each $s_j$, there exists an induced path from $s_j$ to $v_p$.
    
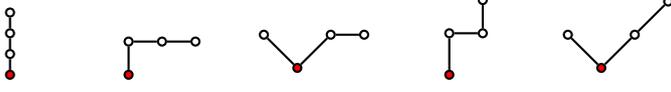
\begin{tikzpicture}[baseline=(current bounding box.center)]
\node[anchor=south] at (0,0)   {\pathverticalfour};
\node[anchor=south] at (2,0)   {\pathHorizontalRed};
\node[anchor=south] at (4,0)   {\pathVshapeHorizontal};
\node[anchor=south] at (6,0)   {\pathLshape};
\node[anchor=south] at (8,0)   {\pathVshape};
\end{tikzpicture}
    \item If there are exactly two vertices of \(P_4\) in layer \(j+1\):
    \begin{enumerate}
        \item Two non-adjacent vertices \(u_1, u_3\): \(s_j\) cannot be adjacent to both (since \(C_4\)-free), hence the induced path \(s_j-u_1-u_2-u_3\) exists.   \mbox{ \pathZigzag } 
        \item Two non-adjacent vertices \(u_1, u_4\): even if \(s_j\) is adjacent to both, the induced path \(s_j-u_1-u_2-u_3\) exists.  \pathCccshape
        \item Two adjacent vertices \(u_1, u_2\): some \(s_j \sim u_1\) (and $s_j \not\sim u_2$), giving \(s_j-u_1-u_2-u_3\).
           \mbox{\pathCshape} \qquad   
           \mbox{\pathLshapeNew} \qquad
        \item Two adjacent vertices \(u_2, u_3\): some \(s_j \sim u_3\), giving \(s_j-u_3-u_2-u_1\). \pathCcshape
    \end{enumerate}

    \item If there are exactly three vertices of \(P_4\) in layer \(j+1\):
    \begin{enumerate}
        \item Three vertices \(u_1, u_2, u_3\): \(s_j \sim u_1\) gives \(s_j-u_1-u_2-u_3\). \pathGammaShape
        \item Three vertices \(u_1, u_3, u_4\): \(s_j \sim u_4\) gives \(s_j-u_4-u_3-u_2\).  \pathPeakShape
    \end{enumerate}

    \item If all four vertices of \(P_4\) are in layer \(j+1\) and \(s_j \sim u_1\), then the induced path \(s_j-u_1-u_2-u_3\) exists, even if \(s_j \sim u_4\). \pathfour

\end{enumerate}

In the above cases, even if \(s_j \in N(v_p)\) (say \(s_j = a_i\)), we obtain an induced path of length 4: \(v_p-a_i-u_l-u_{l+1}-u_{l+2}\).

If \(s_j \in N^{\ge 2}(v_p)\), then at least \(t-1\) vertices adjacent to \(v_p\) in the balloon are not adjacent to the induced paths found.  
If \(s_j \in N^1(v_p)\), even with maximum adjacency to \(P_4\), at most 4 vertices from \(N(v_p)\) are involved, so \(t-4 > 0\) vertices remain.  

Thus, a 2-arm star tree is induced.  
Consequently, \(\chi(Y[N^{\ge 2}(v_p)]) \le dt\), and therefore, the balloon value is at most \(dt + 1\).
\end{proof}

\begin{theorem}\label{TH 2 S+}
For all $d,p \geq 1$ and every $ S(1^{(t-4)}, 4, p)$ , there exists a polynomial $f$ such that for all $t \geq 5$, if a graph \(G \in \mathcal{H}\) is 
$S(1^{(t-4)}, 4, p)$-free and does not contain $K_d(t)$ as a subgraph, then $\chi(G) \leq f(t)$.
\end{theorem}

\begin{proof}
Similar to Theorem~\ref{TH 1 B+}, the proof uses induction on $d$ and substitutes the value $dt + 2$, which Lemma~\ref{lem ba S+} proves is impossible for any $(p,t)$-balloon.
\end{proof}

\noindent
\tikz[baseline=-0.7ex]{\filldraw[black] (0,0) circle (2pt);}\hspace{0.5em}
{\large\bfseries The Class that does not contain $K_3(t)$ as a subgraph}

\vspace{0.5em}

\begin{lemma}\label{lem CLIQ B+}
Let $ p \geq 2 $ and $ t \geq 3 $. If $ G $ is a graph that does not contain $ K_3(t) $ as a subgraph, and $ B^+(p+2, t-1) $-free, then no $ (X, Y) $, $ t $-biclique with value greater than
$
2 + \left(  \sum_{i=0}^{p} t^{i+2} \right)^{(p+3)! \cdot \sum_{i=0}^{p} t^{i}}
$
exists.
\end{lemma}

\begin{proof}
Consider an $(X,Y)$, $t$-biclique. Note that $Y$ does not contain $K_{t,t}$ as a subgraph; 
otherwise, the $(X,Y)$, $t$-biclique would contain $K_3(t)$ as a subgraph.
By Theorem~\ref{theo-scsys}, if $\partial(G) > \left(  \sum_{i=0}^{p} t^{i+2} \right)^{(p+3)! \cdot \sum_{i=0}^{p} t^{i}}$
then \(Y\) contains an induced rooted uniform tree \(H\) with spread \(t$ and height $p$, satisfying $|H| = \sum_{i=0}^{p} t^i$. Therefore $B^+(p+2, t-1)$, being a subgraph of \(H\), appears as an induced subgraph in \(Y\). Thus, we have
\[
\chi(Y) \le 1 +  \left(  \sum_{i=0}^{p} t^{i+2} \right)^{(p+3)! \cdot \sum_{i=0}^{p} t^{i}}.
\]
\end{proof}

\noindent
\tikz[baseline=-0.7ex]{\filldraw[black] (0,0) circle (2pt);}\hspace{0.5em}
{\large\bfseries $t$ is fixed}

\vspace{0.5em}
  
\begin{theorem}\label{TH 3 K3T}
Let $ p \geq 2 $ and $ t \geq 3 $ be fixed integers. 
If $G \in \mathcal{H}$ is a graph that does not contain $K_3(t)$ as a subgraph, 
and $B^{+}(p+2, t-1)$-free, 
then there exists a linear function $f$ such that 
$\chi(G) \le f(\omega(G))$.
\end{theorem}

  \begin{proof}
  
  By substituting the values from Lemma~\ref{lem CLIQ B+} and corollary~\ref{COR ba B+} into Lemma~\ref{lem-mgun}, we obtain the following result.

\[
\chi(G) \le 
\Biggl(\sum_{i=0}^{p-1} t^i\Biggr)
\Bigl(
2 +  \left(  \sum_{i=0}^{p} t^{i+2} \right)^{(p+3)! \cdot \sum_{i=0}^{p} t^{i}}
+ t(2t+9)
\Bigr)
+ t^p (\varphi(t,\omega) + 2).
\]
Therefore, by Theorem~\ref{TH PHI}, we obtain the desired result.
\end{proof}

\medskip
\noindent\textbf{Remark.} The \emph{tree} $B^{+}(p+k, t-1)$, for $k ,p \ge 2 , t\ge 3$, is an induced tree that generalizes $B^{+}(p+2, t-1)$, 
with the only difference being that the pendant vertex is attached to the vertex ($u_k$), which is at distance $k$ from the central vertex.
Let \(\mathcal{F}_k\) ,$k\geq2$, be a subclass of \(\mathcal{H}\) such that in every \((p,t)\)-balloon, $(P,Y)$  where \(v_p\) is the common vertex of \(P\) and \(Y\),  
at least one of the vertices of maximum degree in \(Y[N^{\ge 2}(v_p)]\) lies within distance \(k\) from \(N(v_p)\).
Assume $G \in \mathcal{F}_k$ is \(B^{+}(p+k, t-1)\)-free. Let $z_k$ be a maximum-degree vertex in $N^{\ge 2}(v_p)$ 
residing in the $k$-th layer. Then, analogous to proof in  Lemma~\ref{lem ba B+}, supposing that $Y\left[(N(z_k) \cap N^{\geq 2}(v_p)\right]$
 contains an independent set of size $t$ leads to the existence of a \(B^{+}(p+k, t-1)\) graph, a contradiction.
Therefore, the statement of lemma~\ref{lem ba B+} holds for any graph belonging to $\mathcal{F}_k$ that is $B^{+}(p+k, t-1)$-free. 
As a consequence, Corollary~\ref{COR ba B+} and Theorem~\ref{TH 1 B+} also follow.

\end{document}